\documentclass[11pt]{amsart}

\usepackage{amsfonts,amsmath,amssymb,amscd}

\usepackage{amsthm}
\usepackage{latexsym}
\usepackage[dvips]{graphicx}
\usepackage[dvips]{psfrag}
\usepackage[dvips]{color}
\usepackage{xypic}
\usepackage[abs]{overpic}
\usepackage{caption}
\usepackage[all]{xy}
\usepackage{pinlabel}

\theoremstyle{plain}
\newtheorem*{theorem*}{Theorem}
\newtheorem*{lemma*} {Lemma}
\newtheorem*{corollary*} {Corollary}
\newtheorem*{proposition*}{Proposition}
\newtheorem*{conjecture*}{Conjecture}
\newtheorem{theorem}{Theorem}[section]
\newtheorem{lemma}[theorem]{Lemma}

\theoremstyle{remark}

\theoremstyle{definition}

\newtheoremstyle{citing}
  {}
  {}
  {\itshape}
  {}
  {\bfseries}
  {.}
  {.5em}
{\thmnote{#3}}

\theoremstyle{citing}

\numberwithin{equation}{section}

\begin{document}

\title[The disk complex and 2-bridge knots]{The disk complex and 2-bridge knots}

\author{Sangbum Cho}\thanks{The first-named author is supported in part by Basic Science Research Program through the National Research Foundation of Korea (NRF-2015R1A1A1A05001071) funded by the Ministry of Science, ICT and Future Planning.}

\address{
Department of Mathematics Education \newline
\indent Hanyang University, Seoul 133-791, Korea}
\email{scho@hanyang.ac.kr}

\author{Yuya Koda}\thanks{The second-named author is
supported in part by the Grant-in-Aid for Scientific Research (C),
JSPS KAKENHI Grant Number 17K05254.}

\address{
Department of Mathematics \newline
\indent Hiroshima University, 1-3-1 Kagamiyama, Higashi-Hiroshima, 739-8526, Japan}
\email{ykoda@hiroshima-u.ac.jp}


\date{\today}

\begin{abstract}
We give an alternative proof of a result of Kobayashi and Saeki that every genus one $1$-bridge position of a non-trivial $2$-bridge knot is a stabilization.
\end{abstract}

\maketitle

\section{Introduction}
\label{sec:introduction}
A {\it genus one $1$-bridge position} of a knot $K$ in $S^3$, simply a $(1, 1)$-position of $K$, is a splitting of $S^3$ into two solid tori such that $K$ intersects each of the solid tori in a properly embedded trivial arc.
Similary, a {\it $2$-bridge position} of a knot $K$ in $S^3$ is a splitting of $S^3$ into two $3$-balls such that $K$ intersects each of the $3$-balls in a pair of properly embedded trivial arcs.
Here, a properly embedded arc $\alpha$ is said to be {\it trivial} if there exists an arc $\beta$ in the boundary such that $\alpha \cup \beta$ forms a loop that bounds a disk.
A knot which admits a $2$-bridge position is called a {\it $2$-bridge knot}.
Let $\alpha$ be one of the four arcs in a $2$-bridge position of $K$, that is, $\alpha$ is a trivial arc in one of the two $3$-balls, say $B$.
Let $N(\alpha)$ be a regular neighborhood of $\alpha$ in $B$.
Then the splitting of $S^3$ into the two solid tori $\overline{B - N(\alpha)}$ and $\overline{S^3 - B} \cup N(\alpha)$ turns out to be a $(1, 1)$-position of $K$, which we call a {\it stabilization} of the $2$-bridge position.

It is known that the $2$-bridge position of a $2$-bridge knot $K$ is unique up to equivalence, proved by H. Schubert \cite{Sch56}, so there are exactly four equivalence classes of stabilization depending on the choice of one of the four arcs.
Further T. Kobayashi and O. Saeki showed that any $(1, 1)$-position of a $2$-bridge knot can be obtained in this way, which is the main result of \cite{KS00} stated as follows.

\begin{theorem}
\label{thm:main_theorem}
Every $(1, 1)$-position of a non-trivial $2$-bridge knot is a stabilization.
\end{theorem}

In this work, we give an alternative proof of Theorem \ref{thm:main_theorem}.
The key idea is to consider the $2$-fold cover $L$ of $S^3$ branched along a $2$-bridge knot $K$ and the covering involution $\phi$ of $L$ over $S^3$.
It is well known that $L$ is a lens space and the preimages of the two solid tori of a $(1, 1)$-position of $K$ are genus two handlebodies, say $V$ and $W$.
Each is invariant under the involution $\phi$, as is the preimage of $K$.
We will construct a simplicial complex $\mathcal{PT}(V)$ for the handlebody $V$, called the {\it primitive tree}.
Then $\phi$ defines a simplicial involution of $\mathcal{PT}(V)$ and the existence of a fixed point of the involution on $\mathcal{PT}(V)$ enables us to find some special meridian disks on the two solid tori of the $(1, 1)$-position.
Consequently, we find a $2$-bridge position of $K$ which stabilizes to the original $(1, 1)$- position.

In Section \ref{sec:The disk complex}, we introduce a well-known simplicial complex, called the {\it non-separating disk complex} for a genus-$2$ handlebody $V$.
When $V$ is one of the handlebodies of a genus-$2$ Heegaard splitting for a lens space, the non-separating disk complex for $V$ admits a special subcomplex, called the {\it primitive disk complex}.
In Section \ref{sec:The primitive tree}, the combinatorial structure of the primitive disk complex for each lens space is described, which was done in the previous works \cite{Cho13}, \cite{CK16} and \cite{CK17}.
From the primitive disk complex,  we construct the {\it primitive tree} and introduce some properties of the simplicial automorphisms of the complex that we need.
The proof of Theorem \ref{thm:main_theorem} will be given in the final section.

We use the standard notation $L = L(p, q)$ for a lens space, where its first homology group $H_1(L)$ is the cyclic group of order $|p|$.
We refer \cite{Rol76} to the reader.
The integer $p$ can be assumed to be positive, and any two lens spaces $L(p, q)$ and $L(p', q')$ are diffeomorphic if and only if $p = p'$ and  $q'q^{\pm 1} \equiv \pm 1 \pmod p$.
Thus, we will assume $1 \leq q \leq p/2$.
For convenience, we will not distinguish a disk (or union of disks) and a diffeomorphism from their isotopy classes in their notation.
We will denote by $N(X)$ and $\overline{X}$ a regular neighborhood of $X$ and the closure of $X$ for a subspace $X$ of a space, where the ambient space will be clear from the context.

Finally we remark that the key idea of this work came from Darryl McCullough.
The authors are deeply grateful to him for his valuable advice and comments.

\section{The non-separating disk complex of a genus-$2$ handlebody}
\label{sec:The disk complex}

The {\it non-separating disk complex} for a genus-$2$ handlebody $V$, denoted by $\mathcal{D}(V)$, is a simplicial complex whose vertices are the isotopy classes of essential non-separating disks in $V$ such that a collection of $k+1$ vertices spans a $k$-simplex if and only if it admits a collection of representative disks which are pairwise disjoint.
It is easy to see that $\mathcal{D}(V)$ is $2$-dimensional and every edge of $\mathcal D(V)$ is contained in infinitely but countably many $2$-simplices.
In \cite{McC91}, it is proved that $\mathcal D(V)$ and the link of any vertex of $\mathcal D(V)$ are all contractible.
Thus, the dual complex of $\mathcal D(V)$ is a tree, which we call the {\it dual tree} of $\mathcal D(V)$ simply. The dual tree is the subcomplex of the barycentric subdivision of $\mathcal D(V)$ spanned by the barycenters of the $1$-simplices and $2$-simplices.
See Figure \ref{fig:disk_complex}.

\begin{center}
\labellist
 \pinlabel {$D$} [B] at 13 140
 \pinlabel {$E$} [B] at 345 25
 \pinlabel {$\Delta_1$} [B] at 60 80
 \pinlabel {$\Delta_2$} [B] at 125 55
 \pinlabel {$\Delta_3$} [B] at 172 95
 \pinlabel {$\Delta_4$} [B] at 180 50
 \pinlabel {$\Delta_5$} [B] at 242 55
 \pinlabel {$\Delta_6$} [B] at 290 70
 \endlabellist
\includegraphics[width=7.5cm]{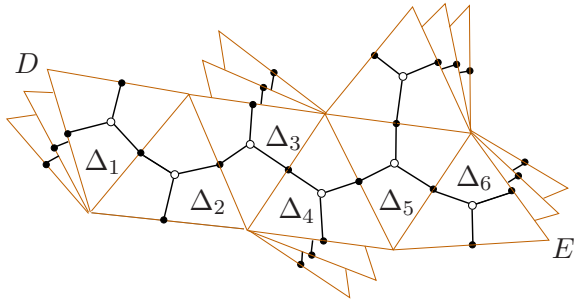}
\captionof{figure}{A portion of the disk complex $\mathcal{D}(V)$ and the dual complex, and the corridor connecting $D$ and $E$.}
\label{fig:disk_complex}
\end{center}

Let $D$ and $E$ be vertices of $\mathcal{D}(V)$ that are not adjacent to each other.
In the barycentric subdivision of $\mathcal{D}(V)$, the links of the vertices $D$ and $E$ are disjoint trees, and hence there exists a unique shortest path in the dual tree of $\mathcal D(V)$ connecting the two links.
Let $v_1,$ $w_1$, $v_2$, $w_2, \ldots, v_{n-1}, w_{n-1}$, $v_n$ be the sequence of the consecutive vertices of this path.
We note that $v_i$ is the barycenter of a $2$-simplex, denoted by $\Delta_i$, for each $i \in \{1, 2, \cdots n\}$.
The full subcomplex of $\mathcal{D}(V)$ spanned by the vertices of $\Delta_1$, $\Delta_2, \ldots, \Delta_n$ is called the {\it corridor} connecting $D$ and $E$, and we denote it just by the sequence $\{\Delta_1 , \Delta_2 , \ldots, \Delta_n \}$.
The vertices $D$ and $E$ are the vertices of $\Delta_1$ and $\Delta_n$ respectively, which are not contained in the edges $\Delta_1 \cap \Delta_2$ and $\Delta_{n-1} \cap \Delta_n$ respectively.

\section{The primitive trees}
\label{sec:The primitive tree}

We denote by $(V, W; \Sigma)$ a genus-$2$ Heegaard splitting of a lens space $L = L(p, q)$.
That is, $V$ and $W$ are genus-$2$ handlebodies such that $V \cup W = L$ and $V \cap W = \partial V = \partial W = \Sigma$, a genus-$2$ closed orientable surface in $L$.
A disk $E$ properly embedded in $V$ is said to be {\it primitive} if there exists a disk $E'$ properly embedded in $W$ such that the two loops $\partial E$ and $\partial E' $ intersect transversely in a single point.
We call such a disk $E'$ a {\it dual disk} of $E$, which is also primitive in $W$ having a dual disk $E$.
The pair of any two disjoint, non-isotopic primitive disks $D$ and $E$ in $V$ is called simply a {\it primitive pair}, and if a disk $E'$ is a dual disk of both $D$ and $E$, then $E'$ is called a {\it common dual disk} of the pair.
Primitive disks are necessarily non-separating, and both $W \cup N(E)$ and $V \cup N(E')$ are solid tori.

The {\it primitive disk complex} $\mathcal P(V)$ for the splitting is then defined to be the full subcomplex of $\mathcal D(V)$ spanned by the vertices of primitive disks.
For each primitive disk $E$ in $V$, it is easy to find infinitely many non-isotopic primitive disks in $V$ disjoint from $E$, so each vertex of $\mathcal P(V)$ has infinite valency for any lens space.
The primitive disk complex $\mathcal P(W)$ for $W$ is defined in the same way, which is isomorphic to $\mathcal P(V)$, since it is known that any two genus-$2$ Heegaard splittings of a lens space are isomorphic to each other (see \cite{Bon83} and \cite{BO83}).

The combinatorial structure of $\mathcal P(V)$ for each lens space was fully studied in \cite{Cho13}, \cite{CK16} and \cite{CK17}.
We describe it as follows. To make the statement simple, we will say that an edge of $\mathcal P(V)$ is of {\it type-$0$} ({\it type-$1$, type-$2$,} respectively) if, up to isotopy, a primitive pair representing the end vertices of the edge has no common dual disk (has a unique common dual disk, has exactly two common dual disks which form a primitive pair in $W$, respectively).

\begin{lemma}[Theorem 4.5 \cite{CK16}]
Let $L = L(p, q)$ be a lens space with $1 \leq q \leq p/2$, and let $(V, W; \Sigma)$ be a genus-$2$ Heegaard splitting of $L$.
If $p \equiv \pm 1 \pmod q$, then the primitive disk complex $\mathcal P(V)$ is contractible, and we have one of the following cases.
\begin{enumerate}
\item If $q \neq 2$ and $p \neq 2q + 1$, then $\mathcal P(V)$ is a tree, and every vertex has infinite valency.
    In this case,
    \begin{enumerate}
    \item if $p=2$ and $q=1$, then every edge is of type-$2$.
    \item if $p \geq 4$ and $q=1$, then every edge is of type-$1$.
    \item if $q \neq 1$, then every edge is of either type-$0$ or type-$1$, and infinitely many edges of type-$0$ and of type-$1$ meet in each vertex.
    \end{enumerate}
\item If $q = 2$ or $p=2q+1$, then $\mathcal P(V)$ is $2$-dimensional, and every vertex meets infinitely many $2$-simplices.
    In this case,
    \begin{enumerate}
    \item if $p = 3$, then every edge is of type-$1$ and is contained in a unique $2$-simplex.
    \item if $p = 5$, then every edge is of either type-$0$ or type-$1$. Among the three edges of a $2$-simplex, exactly one is of type-$1$ and the remaining two are of type-$0$. Every edge of type-$0$ is contained in exactly two $2$-simplices, while every edge of type-$1$ in a unique $2$-simplex.
    \item if $p \geq 7$, then every edge is of either type-$0$ or type-$1$. Among the three edges of a $2$-simplex, exactly one is of type-$1$ and the remaining two are of type-$0$.
        Every edge of type-$0$ is contained in a unique $2$-simplex.
        Every edge of type-$1$ is contained in a unique $2$-simplex or in no $2$-simplex.
    \end{enumerate}
\end{enumerate}
If $p \not\equiv \pm 1 \pmod q$, then $\mathcal P(V)$ is not connected, and it consists of infinitely many tree components.
All the tree components are isomorphic to each other. Any vertex of $\mathcal P(V)$ has infinite valency, and further, infinitely many edges of type-$0$ and of type-$1$ meet in each vertex.
\label{lem:structure}
\end{lemma}

Figure \ref{fig:shape} illustrates a small portion of the primitive disk complex $P(V)$ for each case.
The numbers $0$ or $1$ in the figure indicate the type of the edges.

\begin{center}
\labellist
 \pinlabel {(1) - (a), (b), (c)} [B] at 90 375
 \pinlabel {(2) - (a)} [B] at 400 380
 \pinlabel {(2) - (b)} [B] at 90 173
 \pinlabel {(2) - (c)} [B] at 400 173
 \pinlabel {(3) A portion of $\mathcal P(V)$ with bridges for the case of $p \not\equiv \pm 1 \pmod q$} [B] at 240 5

 \pinlabel {{\small $1$}} [B] at 341 515
 \pinlabel {{\small $1$}} [B] at 358 478
 \pinlabel {{\small $1$}} [B] at 377 512
 \pinlabel {{\small $1$}} [B] at 430 509
 \pinlabel {{\small $1$}} [B] at 464 511
 \pinlabel {{\small $1$}} [B] at 449 476
 \pinlabel {{\small $1$}} [B] at 403 490
 \pinlabel {{\small $1$}} [B] at 383 463
 \pinlabel {{\small $1$}} [B] at 420 464
 \pinlabel {{\small $1$}} [B] at 377 448
 \pinlabel {{\small $1$}} [B] at 358 416
 \pinlabel {{\small $1$}} [B] at 395 419
 \pinlabel {{\small $1$}} [B] at 410 421
 \pinlabel {{\small $1$}} [B] at 427 452
 \pinlabel {{\small $1$}} [B] at 446 424

 \pinlabel {{\small $1$}} [B] at 33 331
 \pinlabel {{\small $1$}} [B] at 72 313
 \pinlabel {{\small $1$}} [B] at 107 308
 \pinlabel {{\small $1$}} [B] at 141 312
 \pinlabel {{\small $1$}} [B] at 38 275
 \pinlabel {{\small $1$}} [B] at 83 275
 \pinlabel {{\small $1$}} [B] at 126 275
 \pinlabel {{\small $1$}} [B] at 170 275
 \pinlabel {{\small $1$}} [B] at 38 250
 \pinlabel {{\small $1$}} [B] at 77 261
 \pinlabel {{\small $1$}} [B] at 134 260
 \pinlabel {{\small $1$}} [B] at 175 251
 \pinlabel {{\small $1$}} [B] at 75 210
 \pinlabel {{\small $1$}} [B] at 107 217
 \pinlabel {{\small $1$}} [B] at 140 212
 \pinlabel {{\small $1$}} [B] at 177 201

 \pinlabel {{\small $0$}} [B] at 11 310
 \pinlabel {{\small $0$}} [B] at 34 308
 \pinlabel {{\small $0$}} [B] at 51 293
 \pinlabel {{\small $0$}} [B] at 71 292
 \pinlabel {{\small $0$}} [B] at 90 285
 \pinlabel {{\small $0$}} [B] at 112 289
 \pinlabel {{\small $0$}} [B] at 131 287
 \pinlabel {{\small $0$}} [B] at 150 296
 \pinlabel {{\small $0$}} [B] at 171 296
 \pinlabel {{\small $0$}} [B] at 30 220
 \pinlabel {{\small $0$}} [B] at 52 235
 \pinlabel {{\small $0$}} [B] at 71 236
 \pinlabel {{\small $0$}} [B] at 92 246
 \pinlabel {{\small $0$}} [B] at 109 242
 \pinlabel {{\small $0$}} [B] at 134 243
 \pinlabel {{\small $0$}} [B] at 149 231
 \pinlabel {{\small $0$}} [B] at 173 231
 \pinlabel {{\small $0$}} [B] at 191 220

 \pinlabel {{\small $0$}} [B] at 378 317
 \pinlabel {{\small $0$}} [B] at 358 283
 \pinlabel {{\small $0$}} [B] at 404 293
 \pinlabel {{\small $0$}} [B] at 452 294
 \pinlabel {{\small $0$}} [B] at 469 314
 \pinlabel {{\small $0$}} [B] at 423 269
 \pinlabel {{\small $0$}} [B] at 386 229
 \pinlabel {{\small $0$}} [B] at 417 228
 \pinlabel {{\small $0$}} [B] at 448 219
 \pinlabel {{\small $0$}} [B] at 442 185

 \pinlabel {{\small $1$}} [B] at 345 320
 \pinlabel {{\small $1$}} [B] at 396 311
 \pinlabel {{\small $1$}} [B] at 418 309
 \pinlabel {{\small $1$}} [B] at 437 314
 \pinlabel {{\small $1$}} [B] at 387 268
 \pinlabel {{\small $1$}} [B] at 370 263
 \pinlabel {{\small $1$}} [B] at 427 251
 \pinlabel {{\small $1$}} [B] at 361 227
 \pinlabel {{\small $1$}} [B] at 350 204
 \pinlabel {{\small $1$}} [B] at 399 196
 \pinlabel {{\small $1$}} [B] at 432 230
 \pinlabel {{\small $1$}} [B] at 476 202
 \endlabellist
\includegraphics[width=12cm]{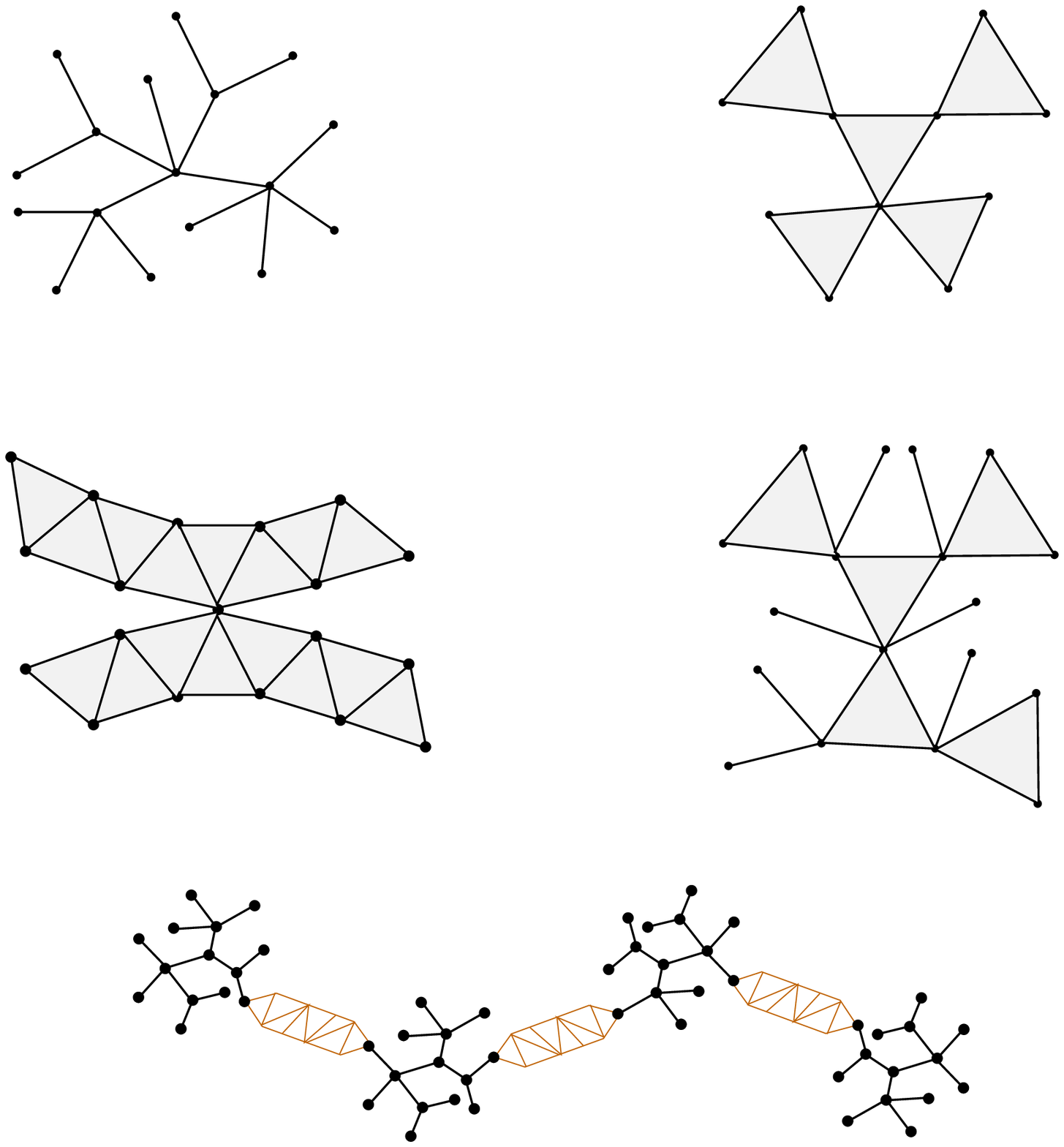}
\captionof{figure}{A portion of primitive disk complex $\mathcal P(V)$ for each lens space.}
\label{fig:shape}
\end{center}

In the case of $p \not\equiv \pm 1 \pmod{q}$, that is, $\mathcal P(V)$ consists of infinitely many tree components, we can define a special kind of corridor in $\mathcal D(V)$, call a {\it bridge}.
A bridge is a corridor connecting the vertices $D$ and $E$ of two primitive disks and contains no vertices of primitive disks except $D$ and $E$.
From the structure of $\mathcal D(V)$ (the dual complex of $\mathcal D(V)$ is a tree), any bridge connects exactly two tree components of $\mathcal P(V)$, and any two tree components of $\mathcal P(V)$ is connected by at most a single bridge.
Note that the union of $\mathcal P(V)$ and the bridges form a contractible subcomplex of $\mathcal D(V)$.
Figure \ref{fig:shape} (3) describes a small portion of four tree components of $\mathcal P(V)$ and three bridges connecting them consecutively.
Further, it was shown in \cite{CK17} that any two bridges are isomorphic to each other, and any two bridges are either disjoint from each other or intersect only in an end vertex (Lemma 2.14 in \cite{CK17}).

The primitive disk complex $\mathcal P(V)$ for a genus-$2$ Heegaard splitting $(V, W; \Sigma)$ of a lens space was used to obtain a finite presentation of the {\it genus-$2$ Goeritz group} of a lens space, denoted by $\mathcal G$.
Here the genus-$2$ Goeritz group $\mathcal G$ is the mapping class group of the splitting, that is, the group of the isotopy classes of orientation preserving diffeomorphisms of the lens space preserving the two handlebodies $V$ and $W$ setwise.
The group $\mathcal G$ acts on the complex $\mathcal P(V)$ simplicially, and the action was fully studied in \cite{Cho13}, \cite{CK16} and \cite{CK17}, and a finite presentation of $\mathcal G$ of each lens space was obtained.
The followings are somewhat technical results on the action of $\mathcal G$ on $\mathcal P(V)$ already developed in the previous works, which we will need in the next section.

\begin{lemma}
Let $L = L(p, q)$ be a lens space with $1 \leq q \leq p/2$, and let $(V, W; \Sigma)$ be a genus-$2$ Heegaard splitting of $L$. Let $\phi$ be an order-$2$ element of the Goeritz group $\mathcal G$, and let $\overline{\phi}$ be the simplicial automorphism of $\mathcal P(V)$ and of $\mathcal P(W)$ defined by $\phi$.
\begin{enumerate}
\item If $\overline{\phi}$ exchanges the end vertices of an edge of type-$2$ of $\mathcal P(V)$, then $\overline{\phi}$ preserves each of the two vertices of $\mathcal P(W)$ of the two common dual disks of the primitive pair representing the end vertices.
\item If $\overline{\phi}$ exchanges the end vertices of an edge of type-$0$ or a bridge, then $\phi$ induces the isomorphism $\phi_*$ of the first homology group $H_1(L)$ such that $\phi_* (x) = qx$.
\end{enumerate}
\label{lem:action}
\end{lemma}

We remark that we have the edge of type-$2$ only when $p = 2$, the edge of type-$0$ only when $q > 2$, and the bridge only when $p \not\equiv \pm 1 \pmod q$.

\begin{proof}
\noindent (1) This follows directly from Lemma 5.2 (1) in \cite{Cho13}, which states that the subgroup $\mathcal G_{\{D \cup E\}}$ of $\mathcal G$ preserving the union of the end vertices $D$ and $E$ of an edge of type-$2$ is the dihedral group $\langle ~\rho, \gamma ~|~ \rho^4 = \gamma^2 = (\rho \gamma )^2 = 1 ~ \rangle$ of order $8$.
The element $\rho$ and $\gamma$ are described in Figure 8 in \cite{Cho13}.
In the figure, $\gamma$ preserves each of the vertices $D$ and $E$ but exchanges the two vertices $D'$ and $E'$ of the two common dual disks of the pair $\{D, E\}$, while $\rho$ exchanges both of $D$, $E$ and $D'$, $E'$.
Thus, an order-$2$ element $\overline{\phi}$ exchanging $D$ and $E$ is either $\rho \gamma$ or $\rho^3\gamma$, and we see that both are preserving each of $D'$ and $E'$.

\noindent (2) The key argument was already provided in the proofs of Lemma 3.13 in \cite{CK16} (for the edge of type-$0$) and Lemma 2.13 in \cite{CK17} (for the bridge).
We sketch the argument briefly.
We first have that there exists an element of $\mathcal G$ exchanging the end vertices of an edge of type-$0$ or a bridge if and only if $q^2 \equiv 1 \pmod p$, from Lemma 5.3 (3) in \cite{CK16} (for the edge of type-$0$) and from Lemma 4.7 (2) in \cite{CK17} (for the bridge).
(In the case of the bridge, for any two tree components $\mathcal T_1$ and $\mathcal T_2$ connected by a bridge, it was shown that there is no element of $\mathcal G$ exchanging $\mathcal T_1$ and $\mathcal T_2$ if $q^2 \not\equiv 1 \pmod p$.)
Thus, with $q^2 \equiv 1 \pmod p$, we can replace $q'$ by $q$, and $\overline{q}$ by $q$ in the proofs of Lemma 3.13 in \cite{CK16} and Lemma 2.13 in \cite{CK17} respectively, where $q'$ is the unique integer satisfying $qq' \equiv \pm 1 \pmod p$ and $\overline{q}$ is one of $q$ and $q'$.

Let $D$ and $E$ be the primitive disks in $V$ representing the vertices of the edge of type-$0$ or a bridge of $\mathcal P(V)$.
Then $V_D = \overline{V - N(D)}$ and $V_E = \overline{V - N(E)}$ are solid tori, that form genus-$1$ Heegaard splittings of $L$ with their exteriors $W_D = \overline{L - V_D}$ and $W_E = \overline{L - V_E}$ respectively.
In the proof of Lemma 3.13 in \cite{CK16}, the core circles of $V_D$ and $V_E$, denoted by $l_D$ and $l_E$, represent the generators $[l_D]$ and $[l_E]$ of $H_1(V_D)$ and $H_1(V_E)$ respectively.
In the proof of Lemma 2.13 in \cite{CK17}, the core circles of $W_D$ and $W_E$, denoted by $l_D$ and $l_E$ again, represent the generators $[l_D]$ and $[l_E]$ of $H_1(W_D)$ and $H_1(W_E)$ respectively.
In either cases, it was shown that $[l_D]$ and $[l_E]$ satisfy $[l_D] = q[l_E]$ in $H_1(L)$ after a suitable choice of orientation of $l_D$ and $l_E$.
The elements $\phi$ sends $E$ to $D$ and hence the isotopy class of $l_E$ to that of $l_D$.
That is,  $\phi$ induces $\phi_*$ such that $\phi_* (1) = q$.
\end{proof}

Now we are ready to construct the primitive tree for each lens space.
Given a lens space $L(p, q)$, $1 \leq q \leq p/2$, with a genus-$2$ Heegaard splitting $(V, W; \Sigma)$, the {\it primitive tree} $\mathcal{PT}(V)$ is defined as follows.

\begin{enumerate}
\item If $p \equiv \pm 1 \pmod q$, and $q \neq 2$ and $p \neq 2q + 1$, then $\mathcal P(V)$ is already a tree. So we just define $\mathcal{PT}(V)$ to be $\mathcal P(V)$. We observe that there are essentially three different kind of trees from Lemma \ref{lem:structure} (1):
    \begin{enumerate}
    \item every edge is of type-$2$ if $p = 2$,
    \item every edge is of type-$1$ if $q = 1$, and
    \item every edge is of either type-$0$ or type-$1$ otherwise.
    \end{enumerate}
\item If $p \equiv \pm 1 \pmod q$, and $q = 2$ or $p = 2q + 1$, then $\mathcal P(V)$ is $2$-dimensional. We have the three cases as stated in Lemma \ref{lem:structure} (2).
    \begin{enumerate}
    \item If $p = 3$, then the primitive disk complex $\mathcal{P}(V)$ deformation retracts to a tree as shown in Figure \ref{fig:shape_2} (2)-(a). We define $\mathcal{PT}(V)$ to be the resulting tree. In the figure, the black vertices are the vertices of $\mathcal{P}(V)$ while the white ones are the barycenters of 2-simplexes of $\mathcal{P}(V)$.
    \item If $p = 5$, then we remove every type-$1$ edge of $\mathcal P(V)$ to get a tree. That is, $\mathcal{PT}(V)$ is the subcomplex of $\mathcal P(V)$ containing only the edges of type-$0$. See Figure \ref{fig:shape_2} (2)-(b).
    \item If $p \geq 7$, then we remove every type-$1$ edge ``contained in a $2$-simplex'' of $\mathcal P(V)$ to get a tree. Then $\mathcal{PT}(V)$ is the subcomplex of $\mathcal P(V)$ containing only the edges of type-$0$, and the edges of type-$1$ that was not contained in any $2$-simplex of $\mathcal P(V)$. See Figure \ref{fig:shape_2} (2)-(c).
    \end{enumerate}
\item If $p \not\equiv \pm 1 \pmod q$, the primitive disk complex $\mathcal{P}(V)$ consists of infinitely many tree components. If there exists a bridge connecting the vertices $D$ and $E$ of $\mathcal P(V)$, we replace the bridge by a new edge connecting $D$ and $E$. Then we define $\mathcal{PT}(V)$ to be the union of $\mathcal P(V)$ and all the new edges came from the bridges. We call such a new edge just a {\it bridge} again. See Figure \ref{fig:shape_2} (3).
\end{enumerate}

We remark that every vertex of the primitive tree $\mathcal{PT}(V)$ is that of the primitive complex $\mathcal P(V)$, the isotopy class of a primitive disk in $V$ except the case of $L(3, 1)$.
For $L(3, 1)$, the tree is bipartite and we have two kind of vertices; the black, the vertex of $\mathcal P(V)$, has infinite valency while the white, the barycenter of a $2$-simplex of $\mathcal P(V)$, has valency three.

\begin{center}
\labellist
 \pinlabel {(2) - (a)} [B] at 95 220
 \pinlabel {(2) - (b)} [B] at 375 220
 \pinlabel {(2) - (c)} [B] at 95 20
 \pinlabel {(3)} [B] at 375 20
 \endlabellist
\includegraphics[width=13cm]{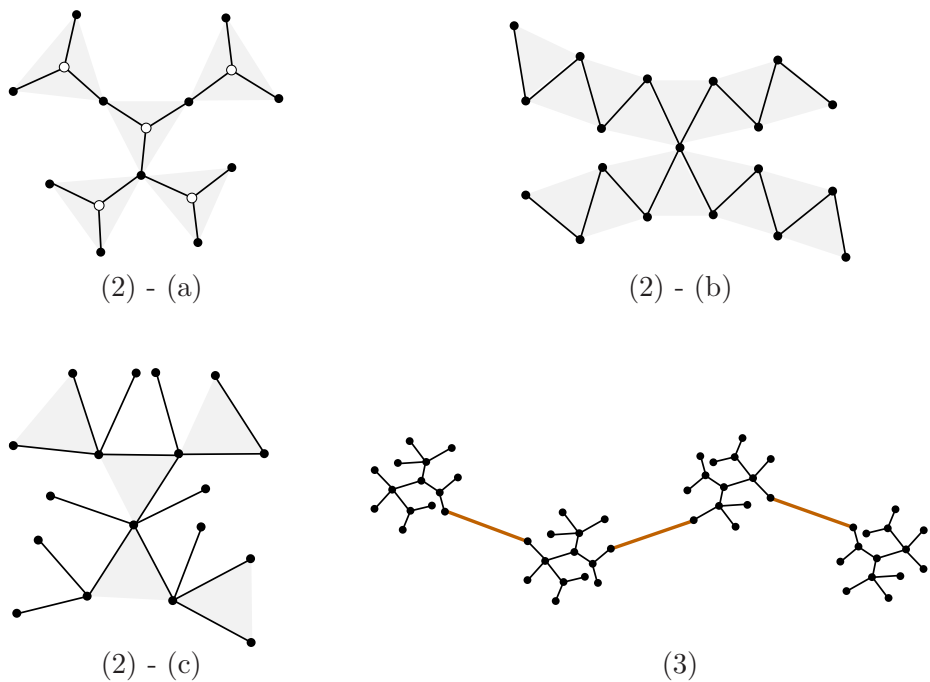}
\captionof{figure}{A portion of primitive tree $\mathcal{PT}(V)$ for each lens space in the cases (2) and (3).}
\label{fig:shape_2}
\end{center}

\section{Proof of the main theorem}
\label{sec:(1, 1)-positions of 2-bridge knots}

In this section, we prove Theorem \ref{thm:main_theorem} using our previous results.
Throughout the section, all the symbols for disks and diffeomorphisms denotes only themselves, not their isotopy classes.
Let $K$ be a knot in $(1, 1)$-position with respect to the Heegaard splitting $(V_0, W_0; \Sigma_0)$ of $S^3$.
That is, each of $V_0$ and $W_0$ is a solid torus that $K$ intersects in a properly embedded trivial arc.
Let $L$ be the 2-fold cover of $S^3$ branched along $K$.
It is well-known that $L$ is a lens space $L(p,q)$ and the preimages of $V_0$ and $W_0$ are genus-$2$ handlebodies $V$ and $W$ that form a genus-$2$ Heegaard splitting of $L$.
Each of $V$ and $W$ is invariant under the covering involution $\phi$ of $L$ over $S^3$, as is the preimage $\widehat{K}$ of $K$.
Thus the isotopy class of $\phi$ is an element of the genus-$2$ Goeritz group $\mathcal G$ of $L$.
We note that $\phi$ induces a homomorphism $\phi_*$ on the first homology group $H_1(L)$ as $\phi_*(x) = -x$.

\begin{lemma}
\label{lem:equivariant cut-and-paste}
Let $V$ and $W$ be the genus-$2$ handlebodies and let $\phi$ be the involution defined as in the above.
Let $D$ be a disk properly embedded in $V$ or in $W$, say in $V$, such that $\partial D = \partial \phi(D)$ or $\partial D$ intersects $\partial \phi(D)$ transversely.
Let $\ell$ be a simple closed curve on $\partial V$ such that $\phi(\ell) = \ell$ and $\ell$ intersects $\partial D$ transversely in a single point.
\begin{enumerate}
\item
There exists a disk $E$ properly embedded in $V$ such that $\partial E = \partial D$ and $E \cap \phi (E)$ contains no loops in the interior.
\item
There exists a disk $F$ properly embedded in $V$ such that $F$ is disjoint from $\phi(F)$, and $\partial F$ intersects $\ell$ in a single point.
\end{enumerate}
\end{lemma}

\begin{proof}
(1) Suppose that $D \cap \phi (D)$ contains loops in the interior.
Let $\gamma$ be an innermost loop of $D \cap \phi (D)$ on $\phi (D)$.
That is, $\gamma$ cuts off a subdisk $D_1$ from $\phi(D)$ such that $D_1 \cap D = \gamma$.
The loop $\gamma$ also bounds a disk $D_0$ on $D$.
Replacing $D$ with $\overline{D \setminus D_0} \cup D_1$ and repositioning by a small isotopy produces a new disk $D_*$ such that $\partial D_* = \partial D$ and $D_* \cap \phi (D_*)$ contains fewer loops than $D \cap \phi (D)$ had.
Repeating the process finitely many times, we finally get a disk $E$ satisfying the required condition.

\noindent (2) Suppose that $\ell$ intersects $\partial D$ in a single point, say $z$. Then $\ell$ intersects $\partial \phi(D)$ only in $\phi(z)$.
By (1), we may assume that $D \cap \phi(D)$ contains no loops in the interior.
Let $\beta$ be an outermost arc of $D \cap \phi(D)$ on $\phi(D)$.
That is, $\beta$ cuts off a subdisk $D_1'$ from $\phi(D)$ such that $D_1' \cap D = \beta$.
Further, we may choose $\beta$ so that $\partial D_1$ does not contain the point $\phi(z)$.
The arc $\beta$ also cuts off a disk $D_0'$ from $D$ that does not contain $z$.
Replacing $D$ with $\overline{D \setminus D'_0} \cup D'_1$ and repositioning by a small isotopy produces a new dual disk $D'_*$ such that $D'_* \cap \phi(D'_*)$ has fewer arcs than $D \cap \phi (D)$ had.
Repeating the process finitely many times,  we finally get a disk $F$ satisfying the required condition.
\end{proof}

\begin{lemma}
\label{lem:key lemma}
Let $V$ and $W$ be the genus-$2$ handlebodies and let $\phi$ be the involution defined as in the above.
After possibly exchanging the names of $V$ and $W$, there exist a primitive disk $J$ in $V$ and its dual disk $R$ in $W$ such that $\phi (J) = J$, and $R$ is disjoint from $\phi(R)$.
\end{lemma}

\begin{proof}
The involution $\phi$ of $L$ defines a simplicial involution, denoted by $\overline{\phi}$, of the primitive tree $\mathcal{PT}(V)$.
Any finite-order automorphism of a tree has a fixed point, and hence there exists a point $v$ of
$\mathcal{PT}(V)$ fixed by the involution $\overline{\phi}$.
The point $v$ is either a vertex of $\mathcal{PT}(V)$ or the barycenter of an edge of $\mathcal{PT}(V)$.

Suppose first that $v$ is a vertex of $\mathcal{PT}(V)$.
Then $v$ is either a vertex of $\mathcal P(V)$ or the barycenter of a $2$-simplex of $\mathcal P(V)$.
The latter case occurs only for the lens space $L(3, 1)$ (see the definition of $\mathcal{PT}(V)$), but in this case,  $\overline{\phi}$ fixes at least one of the three vertices of the $2$-simplex.
Thus, in any cases, $\overline{\phi}$ fixes a vertex of $\mathcal P(V)$, the isotopy class of a primitive disk, say $E$, in $V$.
Since $\phi$ is an involution of $V$, $\partial E$ is isotopic to a loop invariant under $\phi$ (for we may assume that $\phi$ is an isometry with respect to a hyperbolic structure on $\partial V$, and then the unique geodesic in $\partial V$ isotopic to $\partial E$ is $\phi$-invariant).
Once $\partial E$ is $\phi$-invariant, we may replace $E$ with a $\phi$-invariant disk $J$ by Lemma \ref{lem:equivariant cut-and-paste} (1).
Since $J$ is primitive, it has a dual disk $D$ in $W$.
We may assume that $D$ and $\phi(D)$ meet transversely.
By Lemma \ref{lem:equivariant cut-and-paste} (1) and (2), we may replace $D$ with a dual disk $R$ such that $R$ and $\phi(R)$ are disjoint.

Next, suppose that $v$ is the barycenter of an edge of $\mathcal{PT}(V)$, and then the involution $\overline{\phi}$ exchanges the two end vertices of the edge.
We note that $L$ cannot be $L(3,1)$ in this case since a white vertex and a black vertex of $\mathcal{PT}(V)$ cannot be exchanged by the action of $\mathcal{G}$.
Thus, the edge is one of the edges of $\mathcal P(V)$ (the edges of type-$0$, type-$1$ and type-$2$) or bridges.

Suppose that the edge is an edge of type-$0$ or a bridge.
Then we have $\phi_*(x) = qx$ by Lemma \ref{lem:action} (2), but the involution $\phi$ satisfies $\phi_*(x) = -x$, which implies $q = 1$.
This is a contradiction since $\mathcal{PT}(V)$ for $L(p, 1)$ contains neither an edge of type-$0$ nor a bridge.
Therefore, the edge is of type-$1$ or type-$2$.

If the edge is of type-$1$, then the primitive pair of the end vertices of the edge have the unique common dual disk up to isotopy.
We denote by $v'$ the vertex of $\mathcal P(W)$ of the unique common dual disk, which is also fixed by $\overline{\phi}$.
So we can go back to the first case by replacing $V$ with $W$.

If the edge is of type-$2$, that is, $L$ is $L(2, 1)$, then the primitive pair of the end vertices of the edge have exactly two common dual disks up to isotopy, which form a primitive pair in $W$.
By Lemma \ref{lem:action} (1), the involution $\overline{\phi}$ also fixes each of the vertices of $\mathcal P(W)$ of the common dual disks.
We denote by $v'$ the vertex of one of the two common dual disks.
So choosing the vertex $v'$ of $\mathcal P(W)$ and replacing $V$ with $W$, we can go back to the first case.
\end{proof}

\begin{proof}[Proof of Theorem $\ref{thm:main_theorem}$]
It suffices to find meridian disks $J_0$ of $V_0$ and $R_0$ of $W_0$ such that $\partial J_0$ intersects $\partial R_0$ in a single point, $J_0$ intersects $K \cap V_0$ in a single point in its interior, and $R_0$ is disjoint from $K \cap W_0$.
Then letting $B_1$ be the union of $W_0$ with a regular neighborhood of $J_0$ in $V_0$,
and $B_2 = \overline{S^3 \setminus B_1}$, we have a $2$-bridge position for $K$ that stabilizes to the original $(1, 1)$-position.
Let $J$ and $R$ be the primitive disks of $V$ and $W$ respectively obtained in Lemma \ref{lem:key lemma}.
Since $J$ is $\phi$-invariant, we have either $J$ intersects the arc $\widehat{K} \cap V$ in a single point or $J$ contains the arc $\widehat{K} \cap V$. See Figure \ref{fig:covering}.

\begin{center}
\labellist
 \pinlabel {(a)} [B] at 105 8
 \pinlabel {(b)} [B] at 330 8
 \pinlabel {$/\phi$} [B] at 120 150
 \pinlabel {$/\phi$} [B] at 348 150

 \pinlabel {\small $J$} [B] at 120 208
 \pinlabel {\small $J$} [B] at 355 187
 \pinlabel {\small $J_0$} [B] at 148 60
 \pinlabel {\small $J_0$} [B] at 344 109

 \pinlabel {$\widehat{K} \cap V$} [B] at 105 285
 \pinlabel {$\widehat{K} \cap V$} [B] at 333 285
 \pinlabel {$K \cap V_0$} [B] at 145 20
 \pinlabel {$K \cap V_0$} [B] at 373 20
 \pinlabel {\small $\partial R_0$} [B] at 62 46
 \pinlabel {\small $\partial R_0$} [B] at 292 44
 \endlabellist
\includegraphics[width=12cm]{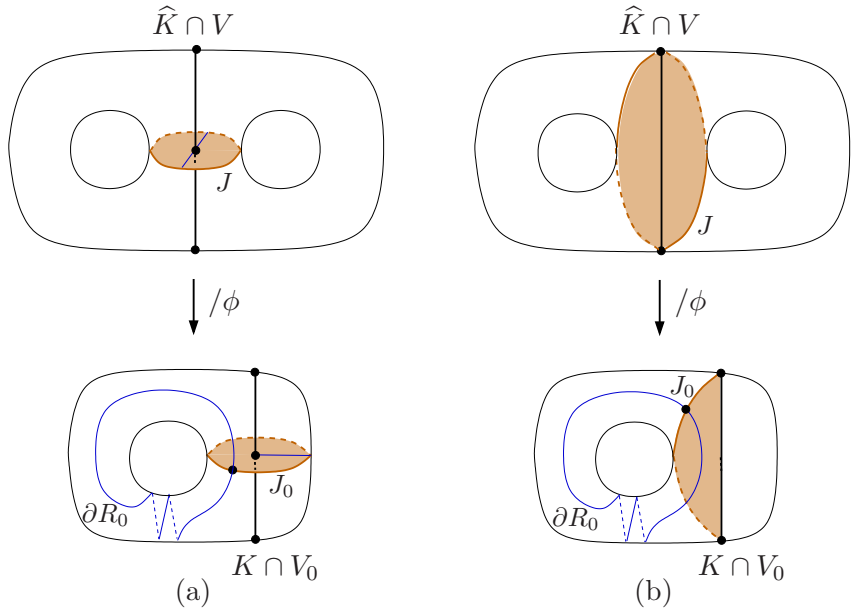}
\captionof{figure}{Two possibilities of the $2$-fold cover $V$ of $V_0$ branched along $K \cap V$.}
\label{fig:covering}
\end{center}

First, consider the case that $J$ intersects the arc $\widehat{K} \cap V$ in a single point.
Then the image $J_0$ of $J$ is a meridian disk of $V_0$ that intersects $K \cap V_0$ in a single point in its interior.
On the other hand the image $R_0$ of $R$ is a disk properly embedded in $W_0$ since $R$ is disjoint from $\phi(R)$.
Further $R_0$ is disjoint from $K \cap W_0$ since $R$ is disjoint from $\widehat{K} \cap W$.
In fact, $R_0$ should be an essential disk since $\partial R_0$ intersects $\partial J_0$ in a single point.
That is, $R_0$ is the desired meridian disk and so we are done.

Next, suppose that $J$ contains $\widehat{K} \cap V$.
Then the image $J_0$ is a bigon in $V_0$ bounded by the arc $K \cap V_0$ and the arc $c$ in $\partial V_0$ that is the image of $\partial J_0$. See Figure \ref{fig:covering} (b).
As in the first case, the image $R_0$ of $R$ is a disk properly embedded in $W_0$ that is disjoint from $K \cap W_0$.
Further, since the arc $c$ intersects $\partial R_0$ once and transversely in its interior,
$R_0$ is non-separating in $W_0$.
In fact, $K \cap W_0$ should intersects $R_0$ otherwise.
Consequently, $R_0$ is a meridian disk of $W_0$.
Now we move the arc $K \cap W_0$ to an arc $d$ by isotopy so that (1) $d$ lies in $\partial W_0 \setminus \partial R$, and (2) the arcs $c$ and $d$ meet only in their end points.
Then $K$ is isotopic to the knot $K_0 = c \cup d$ which lies in $\partial V_0$ and intersects $\partial R_0$ once and transversely.
Then $K_0$ is a $(k, 1)$-torus knot for some $k$, a trivial knot, that is a contradiction.
\end{proof}

\smallskip
\noindent {\bf Acknowledgments.}
Part of this work was carried out while the first and second authors were visiting
Korea Institute for Advanced Study (KIAS) in Seoul, Korea.
They are grateful to the institute and its staff for the warm hospitality.

\bibliographystyle{amsplain}

\end{document}